\def\year{2016}\relax
\documentclass[letterpaper]{article}
\usepackage{aaai16}
\usepackage{times}
\usepackage{helvet}
\usepackage{courier}
\usepackage{wrapfig}

\usepackage{amsmath,longtable,fancyhdr,booktabs,multirow,graphicx,float}
\usepackage{amssymb,xcolor,amsthm,subfigure,textcomp}
\usepackage{theoremref}
\newcommand{\ud}{\mathrm{d}}
\newcommand{\dv}{\vec{d}}
\newcommand{\gammav}{\vec{\gamma}}

\newcommand{\up}{\mathrm}

\newcommand{\mbf}{\mathbf}
\newcommand{\mcal}{\mathcal}
\newcommand{\bb}{\mathbb}

\newcommand{\bs}{\boldsymbol}
\newcommand{\norm}[1]{\left\lVert#1\right\rVert}
\newcommand{\E}[1]{\bb{E}\left[#1\right]}

\newcommand{\T}{\intercal}

\newtheorem{theorem}{Theorem}

\newtheorem{remark}{Remark}

\DeclareMathOperator*{\argmax}{arg\,max}

\title{Bayesian Matrix Completion via Adaptive Relaxed Spectral Regularization}
\author{Yang Song$^\dagger$ \and Jun Zhu$^\ddagger$\\
$^\dagger$Department of Physics, Tsinghua University, yang.song@zoho.com\\
$^\ddagger$Department of Computer Science \& Tech., State Key Lab of Intell. Tech. \& Sys.; CBICR Center;\\
Tsinghua National Lab for Information Science and Tech., Tsinghua University, dcszj@mail.tsinghua.edu.cn
}

\begin{document}
%

\maketitle
\begin{abstract}
\begin{quote}
Bayesian matrix completion has been studied based on a low-rank matrix factorization formulation with promising results. However, little work has been done on Bayesian matrix completion based on the more direct spectral regularization formulation. We fill this gap by presenting a novel Bayesian matrix completion method based on spectral regularization. In order to circumvent the difficulties of dealing with the orthonormality constraints of singular vectors, we derive a new equivalent form with relaxed constraints, which then leads us to design an adaptive version of spectral regularization feasible for Bayesian inference. Our Bayesian method requires no parameter tuning and can infer the number of latent factors automatically. Experiments on synthetic and real datasets demonstrate encouraging results on rank recovery and collaborative filtering, with notably good results for very sparse matrices.
\end{quote}
\end{abstract}

\section{Introduction}
Matrix completion has found applications in many situations, such as collaborative filtering. Let $Z_{m\times n}$ denote the data matrix with $m$ rows and $n$ columns, of which only a small number of entries are observed, indexed by $\Omega \subset [m] \times [n]$.
We denote the possibly noise corrupted observations of $Z$ on $\Omega$ as $P_{\Omega}(X)$, where $P_{\Omega}$ is a projection operator that retains entries with indices from $\Omega$ and replaces others with $0$. The matrix completion task aims at completing missing entries of $Z$ based on $P_{\Omega}(X)$, under the low-rank assumption $\up{rank}(Z) \ll \min(m,n)$. When a squared-error loss is adopted, it can be written as solving:
\begin{align}\tag{P0}
\min_{Z} \frac{1}{2\sigma^2} \norm{P_{\Omega} (X - Z)}_F^2 + \lambda~\up{rank}(Z), \label{rank}
\end{align}
where $\norm{P_{\Omega}(A)}_F^2 = \sum_{(i,j)\in\Omega}a_{ij}^2$; $\lambda$ is a positive regularization parameter; and $\sigma^2$ is the noise variance.

Unfortunately, the term $\up{rank}(Z)$ makes P0 NP-hard. Therefore, the nuclear norm $\norm{Z}_*$
has been widely adopted as a convex surrogate~\cite{fazel2002matrix} to the rank function to turn P0 to a convex problem:
\begin{align}\tag{P1}
\quad \min_{Z} \frac{1}{2\sigma^2} \norm{P_{\Omega} (X - Z)}_F^2 + \lambda \norm{Z}_* \label{nuclear}.
\end{align}

Though P1 is convex, the definition of nuclear norm makes the problem still not easy to solve. Based on a variational formulation of the nuclear norm, it has been popular to solve an equivalent and easier low-rank matrix factorization (MF) form of P1:
\begin{align}
	\min_{A,B} \frac{1}{2\sigma^2} \norm{P_{\Omega} (X - AB^\T)}_F^2 + \frac{\lambda}{2}(\norm{A}_F^2 +\norm{B}_F^2). \label{mf}
\end{align}
Though not joint convex, this MF formulation can be solved by alternately optimizing over $A$ and $B$ for local optima.

As the regularization terms of MF are friendlier than the nuclear norm, many matrix factorization methods have been proposed to complete matrices, including maximum-margin matrix factorization (M\textsuperscript{3}F)~\cite{srebro2004maximum,rennie2005fast} and Bayesian probabilistic matrix factorization (BPMF) \cite{lim2007variational,salakhutdinov2008bayesian}. Furthermore, the simplicity of the MF formulation helps people adapt it and generalize it; e.g., \cite{xu2012nonparametric,xu2013fast} incorporate maximum entropy discrimination (MED) and nonparametric Bayesian methods to solve a modified MF problem.

In contrast, there are relatively fewer algorithms to directly solve P1 without the aid of matrix factorization. Such methods need to handle the spectrum of singular values. These \textit{spectral regularization} algorithms require optimization on a Stiefel manifold~\cite{stiefel1935richtungsfelder,james1976topology}, which is defined as the set of $k$-tuples $(\mbf{u}_1,\mbf{u}_2,\cdots,\mbf{u}_k)$ of orthonormal vectors in $\bb{R}^n$. This is the main difficulty that has prevented the attempts, if any, to develop Bayesian methods based on the spectral regularization formulation.

Though matrix completion via spectral regularization is not easy, there are potential advantages over the matrix factorization approach. One of the benefits is the direct control over singular values. By imposing various priors on singular values, we can incorporate abundant information to help matrix completion. For example, Todeschini et al.~\cite{todeschini2013probabilistic} put sparsity-inducing priors on singular values, naturally leading to hierarchical adaptive nuclear norm (HANN) regularization, and they reported promising results.

In this paper, we aim to develop a new formulation of the nuclear norm, hopefully having the same simplicity as MF and retaining some good properties of spectral regularization. The idea is to prove the \textit{orthonormality insignificance} property of P1. Based on the new formulation, we develop a novel Bayesian model via a sparsity-inducing prior on singular values, allowing various dimensions to have different regularization parameters and automatically infer them. This involves some natural modifications to our new formulation to make it more flexible and adaptive, as people typically do in Bayesian matrix factorization. Empirical Bayesian methods are then employed to avoid parameter tuning. Experiments about rank recovery on synthetic matrices and collaborative filtering on some popular benchmark datasets demonstrate competitive results of our method in comparison with various state-of-the-art competitors. Notably, experiments on synthetic data show that our method performs considerably better when the matrices are very sparse, suggesting the robustness offered by using sparsity-inducing priors.

\section{Relaxed Spectral Regularization}
Bayesian matrix completion based on matrix factorization is relatively easy, with many examples~\cite{lim2007variational,salakhutdinov2008bayesian}. In fact, we can view \eqref{mf} as a \textit{maximum a posterior} (MAP) estimate of a simple Bayesian model, whose likelihood is Gaussian, i.e., for $(i,j) \in \Omega$, $X_{ij} \sim \mcal{N}((AB^\T)_{ij},\sigma^2)$, and the priors on $A$ and $B$ are also Gaussian, i.e., $p(A) \propto \exp (-\lambda \norm{A}^2_F/2)$ and $p(B) \propto \exp (-\lambda \norm{B}_F^2/2)$. It is now easy to do the posterior inference since the prior and likelihood are conjugate.

However, the same procedure faces great difficulty when we attempt to develop Bayesian matrix completion based on the more direct spectral regularization formulation P1. This is because the prior $p(Z) \propto \exp(-\lambda \norm{Z}_*)$ is not conjugate to the Gaussian likelihood (or any other common likelihood). To analyze $p(Z)$ more closely, we can conduct singular value decomposition (SVD) on $Z$ to get $Z = \sum_{k=1}^r d_k \mbf{u}_k \mbf{v}_k^\T$, where $\dv := \{d_k: k\in [r]\}$ are singular values; $U:=\{ \mbf{u}_k: k \in [r] \}$ and $V := \{ \mbf{v}_k : k \in [r] \}$ are orthonormal singular vectors on Stiefel manifolds. Though we can define a factorized prior $p(Z) = p(\dv) p(U)p(V)$, any prior on $U$ or $V$ (e.g., the uniform Haar prior~\cite{todeschini2013probabilistic}) needs to deal with a Stiefel manifold, which is highly nontrivial.

In fact, handling distributions embedded on Stiefel manifolds still remains a largely open problem, though some results~\cite{iyrne2013geodesic,hoff2009simulation,dobigeon2010bayesian} exist in the literature of directional statistics.
Fortunately, as we will prove in Theorem \ref{cool2} that the orthonormality constraints on $U$ and $V$ are not necessary for spectral regularization. Rather,
the unit sphere constraints $\norm{\mbf{u}_k}\leq 1$ and $\norm{\mbf{v}_k}\leq 1$, for all $k \in [r]$, are sufficient to get the same optimal solutions to P1. We call this phenomenon \textit{orthonormality insignificance}. We will call spectral regularization with orthonormality constraints relaxed by unit sphere constraints \textit{relaxed spectral regularization}.

\subsection{Orthonormality insignificance for spectral regularization}
We now present an equivalent formulation of the spectral regularization in P1 by proving its orthornormality insignificance property.

With the SVD of $Z$, we first rewrite P1 equivalently as P1$'$ to show all constraints explicitly:
\begin{align}\tag{P1$'$}
 \min_{\dv, U, V} &\frac{1}{2\sigma^2} \norm{P_{\Omega} \left(X - \sum_{k=1}^r d_k \mbf{u}_k\mbf{v}_k^\T \right)}_F^2 + \lambda \sum_{k=1}^r d_k\label{org} \\
s.t. \quad & d_k \geq 0,\quad \norm{\mbf{u}_k} = 1,\quad \norm{\mbf{v}_k} = 1,\quad \forall k \in [r]\notag\\
\quad &\mbf{u}_i^\T \mbf{u}_j = 0,\quad \mbf{v}_i^\T \mbf{v}_j = 0,\quad \forall i,j \in [r] ~\up{and}~i\neq j\notag,
\end{align}
where $r = \min(m,n)$. Then, we can have an equivalent formulation of P1 as summarized in Theorem~\ref{cool2}, which lays the foundation for the validity of relaxed spectral regularization.
\begin{theorem}\label{cool2} Let $s$ be the optimal value of P1 (P1$'$), and let $t$ be the optimal value of P2 as defined below:
\begin{align}\tag{P2}
\min_{\bs{\alpha}, \bs{\beta}, \dv} ~ & \frac{1}{2\sigma^2}\norm{P_{\Omega}\left(X-\sum_{k=1}^r d_k \bs{\alpha}_k \bs{\beta}_k^\intercal \right)}_{F}^2 + \lambda \sum_{k=1}^r d_k \label{PP} \\
s.t. \quad & d_k \geq 0, \quad \norm{\bs{\alpha}_k}_2 \leq 1, \quad \norm{\bs{\beta}_k}_2 \leq 1, \quad \forall k \in [r], \notag
\end{align}
Then, we have $s = t$. Furthermore, suppose an optimal solution for P2 is $(\dv^*,\bs{\alpha}^*,\bs{\beta}^*)$, then $Z^* = \sum_{k=1}^r d_k^* \bs{\alpha}_k^* \bs{\beta}_k^{*\T}$ is also an optimal solution for P1. Similarly, for any optimal solution $Z^\dagger$ for P1, there exists a decomposition $Z^\dagger = \sum_{k=1}^r d_k^\dagger \bs{\alpha}_k^\dagger {\bs{\beta}_k^\dagger}^\T$ optimal for P2.
\end{theorem}
\begin{proof}[Sketch of the proof]
Let $Z^* = \sum_{k=1}^r d_k^*\bs{\alpha}_k^* {\bs{\beta}_k^*}^\T$ be an optimal solution for P2 with the optimal value $t$. Since P1$'$ is basically the same optimization problem as P2 with stricter constraints, we have $s \geq t$.

Conduct singular value decomposition to obtain $Z^* = \sum_{k=1}^r \sigma_k^* \mbf{u}_k^* {\mbf{v}_k^*}^\T$ and we can prove that $\norm{Z^*}_* = \sum_{k=1}^r \sigma_k^* \leq \sum_{k=1}^r d_k^*$. If $\sum_{k=1}^r \sigma_k^* < \sum_{k=1}^r d_k^*$, then we can plug $Z^*$ into P1 to get a smaller value than $t$, contradicting $s \geq t$. As a result, $\sum_{k=1}^r \sigma_k^* = \sum_{k=1}^r d_k^*$ and $s = t$.

Furthermore, since $s = t$ and plugging $Z^*$ into P1 can lead to a value at least as small as $t$, we conclude that $Z^*$ is also an optimal solution for P1. Let $Z^\dagger$ be any optimal solution for P1, we can also prove that there is a decomposition $Z^\dagger = \sum_{k=1}^{r} d_k^\dagger \bs{\alpha}_k^\dagger {\bs{\beta}_k^\dagger}^\T$ to be an optimal solution for P2.

The formal proof and some remarks are provided in the supplementary material.
\end{proof}

Now we have justified the orthonormality insignificance property of spectral regularization. As a result, P2 serves to be another equivalent form of P1, similar to the role played by MF. This \emph{relaxed spectral regularization} formulation lies somewhere between MF and spectral regularization, since it has more (but easily solvable) contraints than MF and still retains the form of SVD. As discussed before, it is easier to conduct Bayesian inference on a posterior without strict orthonormality constraints, and therefore models on relaxed spectral regularization are our focus of investigation. 

In addition, Theorem~1 can be generalized to \emph{arbitrary} loss besides squared-error loss, which means it is as widely applicable as MF. See Remark~2 in the supplementary material for more details.

\subsection{Adaptive relaxed spectral regularization}
Based on the relaxed spectral regularization formulation in Theorem~1, a Bayesian matrix completion algorithm similar to BPMF can be directly derived. Let the priors of $\bs{\alpha}_k$, $\bs{\beta}_k$ be uniform Haar priors within unit spheres; and the prior of $d_k$'s to be exponential distributions, then the posterior has exactly the same form as P2. Such an algorithm should have similar performance to BPMF.

Instead of building models on P2 exactly, we consider another modified form where each $d_k$ has its own positive regularization parameter $\gamma_k$. Obviously this is a generalization of relaxed spectral regularization and admits it as a special case. We define the adaptive relaxed spectral regularization problem as
\begin{align}\tag{P3}
\min_{ \dv, U, V } &\frac{1}{2\sigma^2} \norm{P_{\Omega} \left(X - \sum_{k=1}^r d_k  \mbf{u}_k\mbf{v}_k^\T \right)}_F^2 + \sum_{k=1}^r \gamma_k d_k\\
s.t. \quad & d_k \geq 0, \quad \norm{\mbf{u}_k}_2 \leq 1,\quad  \norm{\mbf{v}_k}_2 \leq 1, \quad k \in [r].\notag
\end{align}

Such a variation is expected to be more flexible and better at bridging the gap between the nuclear norm and the rank functions, thus being more capable of approximating rank regularization than the standard nuclear norm. Similar ideas arose in~\cite{todeschini2013probabilistic} before and was called \textit{hierarchical adaptive nuclear norm} (HANN). But note that although we propose a similar approach to HANN, our regularization is substantially different because of the relaxed constraints.

However, P3 may be harder to solve than the original P2 due to the difficulty in hyperparameter tuning, since adaptive regularization introduces dramatically more hyperparameters. We will build hierarchical priors for these hyperparameters and derive a Bayesian algorithm for solving P3 and inferring hyperparameters simultaneously in the following section.

\section{Bayesian Matrix Completion via Adaptive Relaxed Spectral Regularization}

\subsection{Probabilistic model}
We now turn P3 into an equivalent MAP estimation task. Naturally, the squared-error loss in P3 corresponds to the negative logarithm of the Gaussian likelihood, $X_{ij} \sim \mcal{N}(\sum_{k=1}^r d_k u_{ki}v_{kj},\sigma^2)$, where $u_{ki}$ denotes the $i$th term in $\mbf{u}_k$; likewise for $v_{kj}$. For priors, we adopt uniform Haar priors on $U$ and $V$ within unit spheres, and exponential priors on $\dv$, as summarized below:
\begin{align*}
	\tilde{p}(\mbf{u}_k) &= \begin{cases}
	1,\quad \norm{\mbf{u}_k} \leq 1\\
	0,\quad \norm{\mbf{u}_k} > 1
		\end{cases},\quad \forall k \in [r]\\
	\tilde{p}(\mbf{v}_k) &= \begin{cases}
	1,\quad \norm{\mbf{v}_k} \leq 1\\
	0,\quad \norm{\mbf{v}_k} > 1
	\end{cases},\quad \forall k \in [r]\\
	\quad	p(d_k\mid \gamma_k) &= \gamma_k e^{-\gamma_k d_k},\quad d_k \geq 0,\quad \forall k \in [r]
	\quad
\end{align*}
where $\tilde{p}$ denotes an unnormalized probability density function (p.d.f.). It can be checked that under this probabilistic model, the negative log posterior p.d.f. with respect to $(\dv, U, V)$ is exactly proportional to P3.

\begin{figure}
	\centering
	\includegraphics[width=\columnwidth]{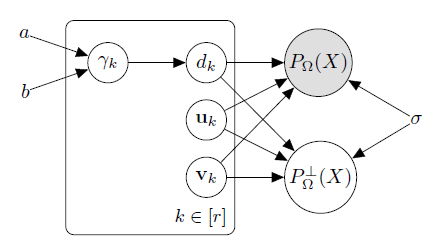}
	\caption{The graphical model for adaptive relaxed spectral regularization, where $P_{\Omega}(A) = \{ a_{ij}\mid (i,j)\in\Omega\}$ and $P_{\Omega}^{\perp}(A) = \{ a_{ij}\mid (i,j)\not \in\Omega\}$.} \label{graph}
\end{figure}

Now we precede to treat regularization coefficients $\gammav := \{ \gamma_k : k \in [r] \}$ as random variables and assume gamma priors on them, i.e., $p(\gamma_k) \propto \gamma_k^{a-1} e^{-b\gamma_k},~ \gamma_k \geq 0,~\forall k \in [r]$. This has two obvious advantages. First, it includes $\gammav$ into the Bayesian framework so that values of these coefficients can be inferred automatically without being tuned as hyperparameters. Second, the prior on $d_k$ after marginalizing out $\gamma_k$'s becomes $p(d_k) = \int_0^\infty p(d_k\mid \gamma_k) p(\gamma_k) \ud \gamma_k = \frac{ab^a}{(d_k+b)^{a+1}}$, which is effectively a Pareto distribution. This distribution has a heavier tail compared to the exponential distribution~\cite{todeschini2013probabilistic}, and is therefore expected to be better at sparsity-inducing~\cite{bach2012optimization}.

The graphical model is shown in Figure~\ref{graph}, where we explicitly separate the observed entries of $X$, i.e., $P_{\Omega}(X)$, and the unobserved ones, i.e., $P_{\Omega}^{\perp}(X)$. Due to the conditional independency structure, we can simply marginalize out $P_{\Omega}^{\perp}(X)$ and get the joint distribution
\begin{align}
	& p( \dv, U, V, \gammav, P_\Omega(X)\mid a,b,\sigma)\notag \\ \label{pdf}
		&\propto \left(\frac{1}{2\sigma^2}\right)^{|\Omega|/2}\exp \left[ -\frac{1}{2\sigma^2} \norm{P_{\Omega}(X - \sum_{k=1}^r d_k \mbf{u}_k\mbf{v}_k^\intercal)}_F^2\right]\notag\\
		&\cdot \prod_{k=1}^r \frac{b^a}{\Gamma (a)} \gamma_k^a e^{-(b+d_k)\gamma_k},
\end{align}
with all the variables implicitly constrained to their corresponding valid domains.

\subsection{Inference}
We now present the GASR (Gibbs sampler for Adaptive Relaxed Spectral Regularization) algorithm to infer the posterior, make predictions, and estimate the hyperparameters via Monte Carlo EM~\cite{casella2001empirical}.

\subsubsection{Posterior Inference}

Let $\mcal{N}(\mu,\sigma^2;a,b)$ denote 
the normal distribution $\mcal{N}(\mu,\sigma^2)$ truncated in $[a,b]$. We infer the posterior distribution $p(\gammav, \dv, U, V \mid a,b,\sigma,P_{\Omega}(X))$ via a Gibbs sampler as explained below:

\textbf{Sample $\gammav$:} The conditional distributions for regularization coefficients $\gammav$ are gamma distributions. We sample $\gammav$ by the formula $\gamma_k \sim \Gamma(a+1;b+d_k), \quad k \in [r]$.

\textbf{Sample $\dv$:} Conditioned on $(\gammav, U, V)$, the distribution for each $d_\alpha$ ($\alpha \in [r]$) is a truncated Gaussian,
$d_\alpha \sim \mathcal{N} (-\frac{B}{A},\frac{\sigma^2}{A};0,\infty),$
where
		$A = \sum_{(i,j)\in \Omega}\left(u_{\alpha i}v_{\alpha j}\right)^2$ and $B = \sum_{(i,j)\in \Omega} \left(\sum_{k\neq \alpha} d_k u_{\alpha i}u_{ki}v_{\alpha j}v_{kj} - X_{ij}u_{\alpha i}v_{\alpha j}\right)+\sigma^2 \gamma_\alpha$.

\textbf{Sample $U$ and $V$:} Given the other variables, the distribution for each element in $\mbf{u}_\alpha$'s (or $\mbf{v}_\alpha$'s) is a truncated Gaussian, 
$u_{\alpha \beta} \sim	\mcal{N}\left(-\frac{D}{C}, \frac{\sigma^2}{C}; - \rho, \rho\right),~ \alpha \in [m],~\beta \in [r],$
where 
$C = \sum_{(\beta,j)\in\Omega}d_\alpha^2 v_{\alpha j}^2$, $D = \sum_{(\beta,j)\in \Omega} \left(\sum_{k\neq \alpha} d_\alpha d_k u_{k\beta}v_{\alpha j}v_{kj} - d_\alpha X_{\beta j} v_{\alpha j}\right)$ and $\rho = \sqrt{1 - \sum_{k \neq \beta}u_{\alpha k}^2}$.
A similar procedure can be derived to sample $v_{\alpha\beta}$ and is therefore omitted here.

The time complexity of this Gibbs sampler is $O(|\Omega|r^2)$ per iteration. Although there is a unified scheme on sampling truncated distributions by cumulative distribution function (c.d.f.) inversion, we did not use it due to numerical instabilities found in experiments. In contrast, simple rejection sampling methods prove to work well.

\subsubsection{Prediction}
With the posterior distribution, we can complete the missing elements using the posterior mean:
\begin{align*}
&\E{P_{\Omega}^\perp(X)\mid P_{\Omega}(X),a,b,\sigma} \\
&= \int \cdots \int \E{P_{\Omega}^\perp(X)\mid P_{\Omega}(X),a,b,\sigma,\gammav, U, V, \dv } \\
&  \cdot p(\gammav, \dv, U, V \mid a,b,\sigma,P_{\Omega}(X)) \ud\gammav \ud \dv \ud U \ud V.
\end{align*}
This integral is intractable. But we can use samples to approximate the integral and complete the matrix. Since we use the Gaussian likelihood, we have
\begin{align*}
\E{P_{\Omega}^\perp(X)\mid P_{\Omega}(X),a,b,\sigma,\gammav, U, V, \dv }
= \sum_{k=1}^r d_k \mbf{u}_k \mbf{v}_k^\T.
\end{align*}
As a result, we can represent missing elements as $x_{ij} = \left<\sum_{k=1}^r d_k u_{ki}v_{kj}\right>,(i,j)\in \Omega^\perp$, which is the posterior sample mean of $P_{\Omega}^\perp (X)$. Here we denote the sample mean for $f(x)$ as $\langle f(x)\rangle := \frac{1}{n}\sum_{i=1}^n f(x_i)$, with $x_i$ being individual samples and $n$ being the number of samples.

\subsubsection{Hyperparameter Estimation}
We choose the hyperparameters $(a,b,\sigma)$ by maximizing model evidence $p(P_{\Omega}(X)\mid a,b,\lambda)$. Since direct optimization is intractable, we adopt an EM algorithm, with latent variable $L := (\dv, U, V, \gammav )$. In order to compute the joint expectation with respect to $P_{\Omega}(X)$ and $L$, we use Monte Carlo EM~\cite{casella2001empirical}, which can fully exploit the samples obtained in the Gibbs sampler.

The expectation of $P_{\Omega}(X)$ and $L$ with respect to $p(L\mid P_{\Omega}(X))$ can be written as
\begin{align}
&\bb{E}_{p(L\mid P_{\Omega}(X))}\left[\ln p(P_{\Omega}(X),L)\right]\notag\\
&=\E{ \ln p(\dv, U, V, \gammav, P_\Omega(X)\mid a,b,\sigma)} \notag \\
		&\approx -|\Omega|\ln \sigma - \frac{1}{2}\left<\norm{P_{\Omega}\left(X - \sum_{k=1}^r d_k \mbf{u}_k \mbf{v}_k^\intercal \right)}_F^2 \right>\notag\\
	&+ \sum_{i=1}^r [a\ln b-\ln\Gamma(a)+a\langle\ln\gamma_i\rangle - b\langle\gamma_i\rangle] + C,\label{E}
\end{align}
where $C$ is a constant. Eq.~\eqref{E} can be maximized with respect to $a,b,\sigma$ using Newton--Raphson method. The fix-point equations are
\begin{align}
	a_{t+1} &= a_t - \frac{\Psi(a_t) - \ln \left( r a_t / \sum_i \langle \gamma_i \rangle\right) - \sum_i \langle \ln \gamma_i \rangle / r}{\Psi'(a_t) - 1 / {a_t}} \nonumber \\
	b_{t+1} &= \frac{r a_{t+1}}{\sum_{i=1}^r \langle \gamma_i \rangle}  \nonumber \\
	\sigma^2 &= \frac{1}{|\Omega|}\left< \norm{P_{\Omega}\left(X - \sum_{k=1}^r d_k \mbf{u}_k \mbf{v}_k^\intercal \right)}_F^2\right>, \nonumber
\end{align}
where $\Psi(x)$ and $\Psi'(x)$ are digamma and trigamma functions respectively. In our experiments, we found the results not very sensitive to the number of samples used in $\langle \cdot \rangle$, so we fixed it to 5.
\section{Experiments}
\begin{figure*} 
\centering

\subfigure{
\includegraphics[width=.32\textwidth]{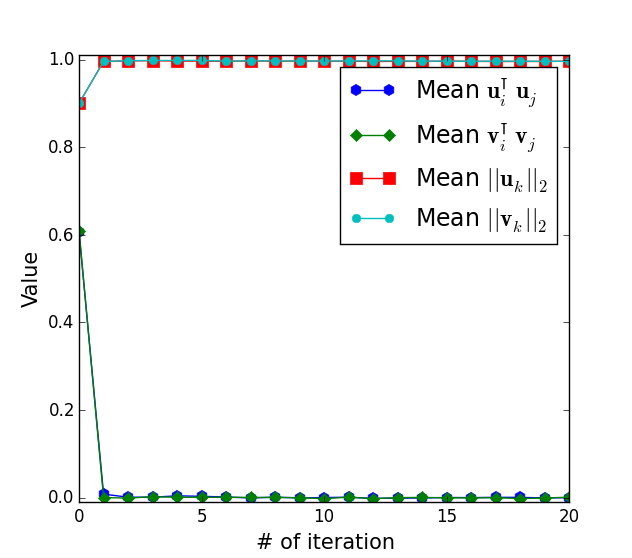}
}
\subfigure{
\includegraphics[width=.32\textwidth]{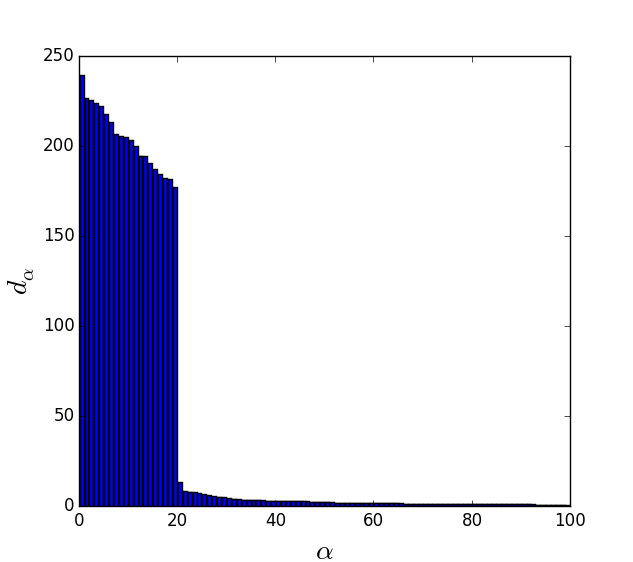}
}
\subfigure{
\includegraphics[width=.32\textwidth]{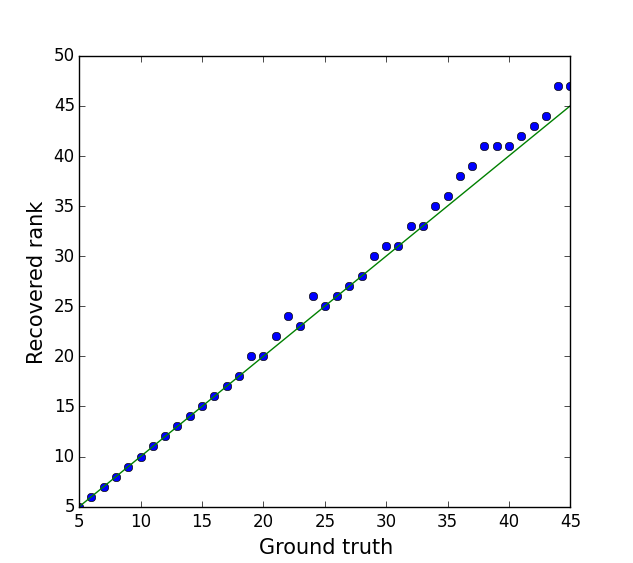}
}
\caption{
(a)\label{on} Experiment on $200\times 200$ synthetic matrix shows that vectors tend to get orthonormalized. We measure the mean values of inner products and norms of all column vectors in $U$ and $V$ for each single iteration. The fluctuations at the beginning are due to initialization. (b) The recovered rank of a $200\times 200$ synthetic matrix with rank 20. Following the principle presented in main text, we conclude that the number of latent factors is 20, matching the ground truth.\label{exam} (c) Rank recovery results on synthetic data, with solid line representing the ground truth.\label{rankrec}} 
\end{figure*}

We now present experimental results on both synthetic and real datasets to demonstrate the effectiveness on rank recovery and matrix completion.

\subsection{Experiments on synthetic data}
We have two experiments on synthetic data, one is for rank recovery, the other is for examining how well the algorithms perform in situations that matrices are very sparse.

In both experiments, we generate standard normal random matrices $A_{m\times q}$ and $B_{n \times q}$ and produce a rank-$q$ matrix $Z = AB^\intercal$. Then we corrupt $Z$ with standard Gaussian noise to get the observations $X$, using a signal to noise ratio of 1.
\subsubsection{Rank recovery}
In this experiment, we set $m = 10 q$ and $n = 10 q$. The algorithm was tested with $q$ ranging from 5 to 45. We set the rank truncation $r$ to $100$, which is sufficiently large for all data. For each matrix $Z$, the iteration number was fixed to 1000 and the result was averaged from last 200 samples (with first 800 discarded as burn-in). We simply initialize our sampler with uniformly distributed $U$ and $V$ with norms fixed to 0.9 and all $\dv$ fixed to zero. We run our Gibbs sampler on all the entries of $X$ to recover $Z$.

In the spectral regularization formulation, we can get the number of latent factors by simply counting the number of nonzero $d_k$'s. However, since our method uses MCMC sampling, it is diffcult to find some $d_k$ to vanish exactly. Instead of counting nonzero elements of $(d_1,d_2,\cdots,d_r)$ directly, we sort the tuple in an ascending order and try to locate $w = \argmax_{k\geq 2} {d_k/d_{k-1}}$ and then discard the set $\{d_k: d_k < d_w\}$. As a result, the recovered rank is $r - w + 1$. The middle panel of Figure~\ref{exam} provides an example about how to determine the number of latent factors.

The results of this experiment are summarized in Figure~\ref{rankrec}(c), showing that the recovered ranks are fairly aligned with the ground truth. Our algorithm performs perfectly well when the true rank is relatively small and slightly worse when the rank gets higher. This may be due to that larger rank requires more iterations for convergence.

We also illustrate how vectors get orthonormalized automatically on one of our synthetic matrices in Figure \ref{on}(a). The orthonormality of vectors are measured by the average values of their $\ell_2$ norms and inner products with each other. Figure~\ref{on}(a) shows that $U$ and $V$ get nearly orthonormalized after only one iteration. This phenomenon indicates that the vectors still tend to get orthonormalized even in the hierarchial Bayesian model.

\subsubsection{Different missing rates}
We generate matrices of different sizes and different missing rates to test the performance of our method, in comparison with BPMF, as it is the only one that can compete with GASR on real datasets, as illustrated in detail in the next section.

The \textit{Root Mean Squared Error} (RMSE) results are listed in Table~\ref{sparse}. The deviations and some additional settings are reported in supplementary material. We can see that GASR is considerably better than BPMF when the observed elements of a matrix are of a small number, demonstrating the robust estimates of GASR via sparsity-inducing priors.

\begin{table}[t]
\caption{Results on different missing rates}
\label{sparse}
\begin{center}
\begin{tabular}{ccccc}
\toprule
Setting & \multicolumn{4}{c}{$m=500,n=500,r=30,q=5$}  \\
\midrule
Missing-Rates & 90\%  & 80\%  & 50\%  & 0\% \\
\midrule
BPMF & $1.6842$ & $0.3210$ & $0.1304$ & $0.0933$\\
GASR & $0.1992$ & $0.1321$ & $0.0841$ & $0.0724$\\
\midrule
Setting & \multicolumn{4}{c}{$m=1000,n=1000,r=50,q=10$}  \\
\midrule
Missing-Rates & 90\%  & 80\%  & 50\%  & 0\% \\
\midrule
BPMF & $0.9422$ & $0.2396$ & $0.1105$ & $0.0859$\\
GASR & $0.2513$ & $0.1688$ & $0.1270$ & $0.1115$\\
\bottomrule
\end{tabular}
\end{center}
\end{table}

\subsection{Collaborative filtering on real datasets}

We test our algorithm on the MovieLens 1M\footnote{MovieLens datasets can be downloaded from http://grouplens.org/datasets/movielens/.} and EachMovie datasets, and compare results with various strong competitors, including max-margin matrix factorization (M\textsuperscript{3}F)~\cite{rennie2005fast}, infinite probabilistic max-margin matrix factorization (iPM\textsuperscript{3}F)~\cite{xu2012nonparametric}, softImpute~\cite{mazumder2010spectral}, softImpute-ALS (``ALS'' for ``alternating least squares'')~\cite{hastie2014matrix}, hierarchical adaptive soft impute (HASI)~\cite{todeschini2013probabilistic} and Bayesian probabilistic matrix factorization (BPMF)~\cite{salakhutdinov2008bayesian}.

The MovieLens 1M dataset contains 1,000,209 ratings provided by 6,040 users on 3,952 movies. The ratings 
are integers from $\{1,2,3,4,5\}$ and each user has at least 20 ratings. The EachMovie dataset consists of 2,811,983 ratings provided by 74,424 users on 1,648 movies. As in~\cite{marlin2004collaborative}, we removed duplicates and discarded users with less than 20 ratings. This left us with 36,656 users. There are 6 possible ratings from 0 to 1 and we mapped them to $\{1,2,\cdots,6\}$.

\textbf{Protocol:} We randomly split the dataset into 80\% training and 20\% test. We further split 20\% training data for validation for M\textsuperscript{3}F, iPM\textsuperscript{3}F, SoftImpute, SoftImpute-ALS and HASI to tune their hyperparameters. BPMF and GASR can infer hyperparameters from training data and thus do not need validation. We measure the performance using both RMSE and \textit{normalized mean absolute error} (NMAE), where NMAE~\cite{goldberg2001eigentaste} is defined as
\begin{align}
	\up{NMAE} = \frac{1}{|\Omega_{\up{test}}|} \frac{\sum_{(i,j)\in\Omega_{\up{test}}}|X_{ij} - Z_{ij}|}{\max(X) - \min(X)},
\end{align}
and $\Omega_{\up{test}}$ is the index set of entries for testing.

\begin{figure} 
\centering
\includegraphics[width = .8\columnwidth]{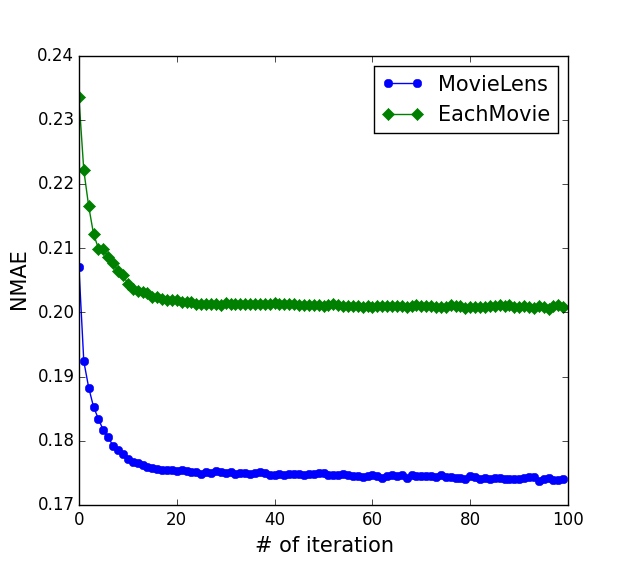} 
\caption{Convergence of NMAE on real datasets. We only show results on our first data partitions.}\label{iter} 
\end{figure}

\begin{table*}[t]
\caption{Experimental results of various methods on the MovieLens and EachMovie datasets.}\label{movielens} 

\begin{center}
\begin{tabular}[htbp]{c|c|c|c|c}
\hline
&
\multicolumn{2}{c|}{\textbf{MovieLens}} & \multicolumn{2}{c}{\textbf{EachMovie}}\\\hline
  \textbf{Algorithm}  & \textbf{NMAE} & \textbf{RMSE} & \textbf{NMAE} & \textbf{RMSE}\\
  \hline
  M\textsuperscript{3}F & $0.1652 \pm 0.0004$ & $0.9644 \pm 0.0012$ & $0.1932 \pm 0.0003$ & $1.4598\pm 0.0019$\\
  iPM\textsuperscript{3}F & $0.1604 \pm 0.0004$ & $0.9386 \pm 0.0024$ & $0.1819 \pm 0.0006$ & $1.3760 \pm 0.0034$\\
  SoftImpute   & $0.1829\pm0.0002$ & $0.9469\pm0.0009$ & $0.2039\pm0.0002$ & $1.2948\pm0.0037$\\
  SoftImpute-ALS & $0.1783\pm0.0001$ & $0.9196\pm0.0013$ & $0.2018\pm0.0004$ & $1.2757\pm0.0008$\\
  HASI & $0.1813\pm 0.0002$ & $0.9444\pm0.0011$ & $0.1992\pm0.0003$ & $1.2612\pm0.0016$\\
  BPMF & $0.1663\pm0.0002$ & $0.8460 \pm 0.0006$ & $0.2012\pm0.0001$ & $1.2363 \pm 0.0007$\\\hline
  GASR  & $0.1673\pm 0.0005$ & $0.8528\pm0.0025$ & $0.1930\pm0.0009$ & $1.2015\pm0.0044$\\\hline
\end{tabular}
\end{center} 
\end{table*}

\textbf{Implementation details:} The number of iterations of our sampler is fixed to 100 and the point estimates of $\E{P_{\Omega}^\perp (X)}$ are taken from the average of all 100 samples. We initialize our algorithm by generating uniformly distributed items for $U$ and $V$ and set all $d_k$ to 0. We also scale the norms of $\mbf{u}_k$ and $\mbf{v}_k$ to 0.9 for initialization. Figure \ref{iter} shows that our sampler converges after a few iterations with this fairly na\"ive initialization.

We use the R package \verb|softImpute| for SoftImpute and SoftImpute-ALS, and use the code provided by the corresponding authors for M\textsuperscript{3}F, iPM\textsuperscript{3}F, HASI and BPMF. The hyperparameters for M\textsuperscript{3}F, iPM\textsuperscript{3}F, SoftImpute, SoftImpute-ALS and HASI are selected via grid search on the validation set. We randomly initialize all methods except HASI
for which the initialization is the result of SoftImpute, as suggested in~\cite{todeschini2013probabilistic}. The results of BPMF are the averages over 100 samples, the same as ours.

For all the algorithms, we set the maximum number of iterations to 100. The rank truncation $r$ for MovieLens 1M and EachMovie is set to 30, where we follow the setting in~\cite{todeschini2013probabilistic} and find in experiments that larger $r$ does not lead to significant improvements.

\textbf{Results:}
Table \ref{movielens} presents the averaged NMAE and RMSE over 5 replications and their standard deviations.\footnote{The NMAE of HASI on MovieLens is slightly different from that in~\cite{todeschini2013probabilistic}, which was $0.172$, still inferior to ours. This may be due to differences in parameter selecting methods.} Overall, we can
see that our GASR achieves superior performance than most of the baselines. More specifically, we have the following observations:

First, GASR is comparable to BPMF, the state-of-the-art Bayesian method for low-rank matrix completion, on the MovieLens dataset; while it outperforms BPMF on the EachMovie dataset, with the observation that EachMovie dataset (97.8\% missing) is sparser than MovieLens (95.8\% missing). On both datasets, GASR also obtains much lower RMSE than iPM\textsuperscript{3}F, a state-of-the-art nonparametric Bayesian method based on IBP~\cite{griffiths2011indian} for matrix completion. Such results demonstrate the promise of performing Bayesian matrix completion via spectral regularization. Furthermore, GASR produces sparser solutions due to its sparsity-inducing priors on $\dv$. The ranks it infers on MovieLens and EachMovie are both $10$ on average, but the numbers of latent factors inferred by iPM\textsuperscript{3}F are both $30$, which is the rank truncation level. It is reported in~\cite{xu2013fast} with a similar setting that the optimal latent dimensions inferred by Gibbs iPM\textsuperscript{3}F, a Gibbs sampling version for iPM\textsuperscript{3}F model without rank truncations, are around 450 for MovieLens and 200 for EachMovie, which are much larger than ours.

Second, compared to HASI, a non-Bayesian method that adopts similar adaptive spectral regularization, and the other non-Bayesian methods based on squared-error losses (i.e., SoftImpute and SoftImpute-ALS), we achieve much better results on both datasets, demonstrating the advantages of Bayesian inference. Moreover, the better performance of HASI over SoftImpute demonstrates the benefit of adaptivity.

Finally, the max-margin based methods (i.e., M\textsuperscript{3}F and iPM\textsuperscript{3}F) have slightly better performance on NMAE but much worse results on RMSE than our GASR. One possible reason is that these methods are based on the maximum-margin criterion, which naturally minimizes absolute errors, while our method (and the others) is based on minimizing a squared-error loss.
Another reason, which may be the most important one, is that both M\textsuperscript{3}F and iPM\textsuperscript{3}F predict integer values while our method (and the others) gives real value predictions. We found that simply rounding these real value predictions to integers can greatly improve the performance on NMAE. For example, our GASR gives NMAE value $0.1569 \pm 0.0006$ and $0.1877 \pm 0.0003$ respectively on MovieLens and EachMovie datasets after rounding predictions to nearest integers.

\section{Conclusions and Discussions}

We present a novel Bayesian matrix completion method with adaptive relaxed spectral regularization. Our method exhibits the benefits of hierarchical Bayesian methods on inferring the parameters associated with adaptive relaxed spectral regularization, thereby avoiding parameter tuning. We estimate hyperparameters using Monte Carlo EM.  
Our Gibbs sampler 
exhibits good performance both in rank inference on synthetic data and collaborative filtering on real datasets.

Our method is based on a new formulation in Theorem \ref{cool2}. These results can be further generalized to other noise potentials with minor effort. For the Gibbs sampler, we can also extend to non-Gaussian potentials
as long as this potential has a regular p.d.f. that enables efficient sampling.

Finally, though we stick to Gibbs sampling in this paper, it would be interesting to investigate other Bayesian inference methods based on Theorem \ref{cool2} since it gets rid of many difficulties related to Stiefel manifolds. Such an investigation may lead to more scalable algorithms with better convergence property. Furthermore, better initialization methods other than uniformly generated random numbers may lead to much faster convergence, e.g., results from several iterations of HASI can usually provide a good starting point.

\subsection{Acknowledgments}
The work was supported by the National Basic Research Program (973 Program) of China (Nos.~2013CB329403, 2012CB316301), National NSF of China (Nos.~61322308, 61332007), Tsinghua National Laboratory for Information Science and Technology Big Data Initiative, and Tsinghua Initiative Scientific Research Program (No.~20141080934). We thank the Department of Physics at Tsinghua University for covering part of the travel costs.
\bibliography{Song.bib}
\bibliographystyle{aaai}

\section*{Supplementary Materials}
\subsection{Proof for Theorem 1}

\begin{proof}
\begin{table*}[ht]
\centering
\caption{Results on different missing rates}
\label{sparse}
\begin{tabular}{ccccc}
\toprule
Setting & \multicolumn{4}{c}{$m=500,n=500,r=30,q=5$}  \\
\midrule
Missing-Rates & 90\%  & 80\%  & 50\%  & 0\% \\
\midrule
BPMF & $1.6842\pm0.1374$ & $0.3210\pm 0.0168$ & $0.1304\pm0.0022$ & $0.0933 \pm 0.0000$\\
GASR & $0.1992\pm 0.0241$ & $0.1321\pm 0.0086$ & $0.0841\pm0.0028$ & $0.0724\pm 0.0036$\\
\midrule
Setting & \multicolumn{4}{c}{$m=1000,n=1000,r=50,q=10$}  \\
\midrule
Missing-Rates & 90\%  & 80\%  & 50\%  & 0\% \\
\midrule
BPMF & $0.9422\pm0.0478$ & $0.2396\pm0.0033$ & $0.1105\pm 0.0013$ & $0.0859 \pm 0.0007$\\
GASR & $0.2513\pm0.0045$ & $0.1688\pm 0.0041$ & $0.1270\pm0.0034$ & $0.1115\pm 0.0057$\\
\bottomrule
\end{tabular}
\end{table*}

Denote $Z = \sum_{k=1}^r d_k \bs{\alpha}_k \bs{\beta}_k^\intercal$ and conduct singular value decomposition to give
\begin{align}
	Z = \sum_{k=1}^r \sigma_k \b{u}_k \b{v}_k^\intercal  = \sum_{k=1}^r d_k \bs{\alpha}_k \bs{\beta}_k^\intercal\label{e0}
\end{align}
where $\sigma_{1:r}$ are singular values of $Z$ and $U = \{\mbf{u}_1,\mbf{u}_2,\cdots,\mbf{u}_m\}$ and $V = \{ \mbf{v}_1,\mbf{v}_2,\cdots,\mbf{v}_n \}$ are corresponding orthogonal matrices. Note that we denote $r = \min(m,n)$ throughout this paper. If the true rank is smaller than $\min(m,n)$, then singular values with indices larger than the rank are assumed to be zero. We will try to prove
\begin{equation}
	\sum_{k=1}^r \sigma_k \leq \sum_{k=1}^r d_k,\label{ieq}
\end{equation}
which actually implies all the assertions in the theorem.

For $\forall i \in [r]$, left multiply equation \eqref{e0} with $\b{u}_i^\intercal$ and right multiply with $\b{v}_i$ to obtain
\begin{align}
	\sigma_i = \sum_{k=1}^r d_k \b{u}_i^\intercal \bs{\alpha}_k \bs{\beta}_k^\intercal \b{v}_i. \label{e1}
\end{align}
Since $U$ and $V$ are orthogonal matrices with full ranks in a singular value decomposition, we can regard column vectors of $U$ and $V$ to be eigenbases of space $\bb{R}^m$ and $\bb{R}^n$. Hence it is natural to obtain
\begin{align}
	\bs{\alpha}_i = \sum_{j = 1}^m x_{ij} \mbf{u}_j,\quad
	\bs{\beta}_i = \sum_{j=1}^n y_{ij} \mbf{v}_j,
\end{align}
where $\forall i\in [r],\quad \sum_{j=1}^m x_{ij}^2 \leq 1$ and $\sum_{j=1}^n y_{ij}^2 \leq 1$.

We then rewrite Eq.~\eqref{e1} to give
\begin{align}
	\sigma_i = \sum_{k=1}^r d_k x_{ki}y_{ki}.
\end{align}
As a result,
\begin{align}
	\sum_{i=1}^r \sigma_i &= \sum_{i=1}^r\sum_{k=1}^r d_k x_{ki}y_{ki}\notag = \sum_{k=1}^r d_k \sum_{i=1}^r x_{ki}y_{ki}\notag
	\\\notag
	&\leq \sum_{k=1}^r d_k \left(\sum_{i=1}^r x_{ki}^2 \right)^{1/2}\left(\sum_{i=1}^r y_{ki}^2\right)^{1/2}\\\notag
	&\leq \sum_{k=1}^r d_k \left(\sum_{i=1}^m x_{ki}^2 \right)^{1/2}\left(\sum_{i=1}^n y_{ki}^2\right)^{1/2}\\
	&\leq \sum_{k=1}^r d_k,
\end{align}
which means for any valid tuples of $(d_{1:r},\bs{\alpha}_{1:r},\bs{\beta}_{1:r})$, replacing $\sum_{k=1}^r d_k \bs{\alpha}_k \bs{\beta}_k^\intercal$ with $ \sum_{k=1}^r \sigma_k \b{u}_k \b{v}_k^\intercal$ in P2, according to \eqref{ieq}, will not make the solution worse. This indicates that there is at least one optimal solution of P2 having the form of singular value decompositions like those in P1$'$.

As a result, we always have $s \leq t$, because for any optimal solution of P2, we can get an SVD form compatible to the constraints of P1$'$ with an objective value not larger. However, considering the fact that P2 is basically the same problem as P1$'$ with looser constraints, we conclude that $t \leq s$. Following the reasoning above we get $s \leq t$ and $s \geq t$, which exactly means $s = t$.

Suppose we have got the optimal solutions of P2, which is denoted as $(d_{1:r}^*,\bs{\alpha}_{1:r}^*,\bs{\beta}_{1:r}^*)$. We assert that $Z^* = \sum_{k=1}^r d_k^* \bs{\alpha}_k^* \bs{\beta}_k^{*\T}$ is the optimal solution of P1, because plugging $Z^*$ into P1 will yield a value not greater than $t$. Since $s = t$ and $s$ is the minimum possible value of P1, we conclude that plugging $Z^*$ into P1 gets the value $s$, which means $Z^*$ is the optimal solution for P1.

Similarly, suppose that the optimal solution of P1 is $Z^\dagger$, we compute its singular value decomposition to get $Z^\dagger = \sum_{k=1}^r \sigma_k^\dagger \mbf{u}_k^\dagger {\mbf{v}_k^\dagger}^\T$. Then plugging $(\sigma^\dagger_{1:r},\mbf{u}_{1:r}^\dagger,\mbf{v}_{1:r}^\dagger)$ into P2 will give the value $s$. Since $s = t$, we conclude that $Z^\dagger$ is an optimal solution for P2.

Note that it is practically very difficult for $\sum_{k=1}^r \sigma_k = \sum_{k=1}^r d_k$ to hold as this requires $\sum_{i=1}^r x_{ki}^2 = \sum_{i=1}^r y_{ki}^2 = 1$ and $x_{ki} = y_{ki}$, for all $k,i \in[r]$. This means that conducting singular value decomposition to any $Z = \sum_{k=1}^r d_k \bs{\alpha}_k \bs{\beta}_k^\T$ and substituting singular values and vectors into P2 can typically get a better result, which indicates that $\bs{\alpha}_{1:r}$ and $\bs{\beta}_{1:r}$ will nearly always get orthonormalized automatically under the unit sphere constraints.

\end{proof}

\begin{remark}\label{rmk}
The optimal solutions for P2 are not necessarily unique. As a result, the optimal $\bs{\alpha}_{1:r}$ and $\bs{\beta}_{1:r}$ for P2 are not always orthonormalized, though orthonormal vectors provide a solution. However, what matters is not the orthonormality of $\bs{\alpha}_{1:r}$ and $\bs{\beta}_{1:r}$, but the equivalence of optimal solution $Z = \sum_{k=1}^r d_k \bs{\alpha}_k \bs{\beta}_k^\intercal$. Theorem~1 asserts that P1 and P2 produce the same set of optimal matrix completion results. As a result, the MAP problem constructed according to P2 is anticipated to function similarly as the one constructed from P1.

Note that the condition for strict equality in $\sum_{k=1}^r \sigma_k^* \leq \sum_{k=1}^r d_k^*$ is practically very hard to satisfy.
\end{remark}

\begin{remark}
This special relationship between P1 and P2 in Theorem~1 can be generalized to other forms of noise potentials besides the squared-error loss as well as the max-margin hinge loss used in MMMF, as we do not need the property of $\norm{\cdot}_F$ in our proof. The theorem should still hold if we replace $\norm{\cdot}_F$ with $\norm{\cdot}_1$, $\norm{\cdot}_\infty$, etc. 
\end{remark}
\subsection{Detailed Experimental Results for Different Missing Rates}

In this experiment, we run both BPMF and GASR for 100 iterations and average all 100 samples to produce the final result. The average RMSE and corresponding deviations on 3 randomly generated datasets are reported in Table~\ref{sparse}. 

\end{document}